\documentclass{article}
\usepackage[utf8]{inputenc}
\usepackage[left=1in,top=1in,right=1in,bottom=1in]{geometry}
\usepackage[linktocpage=true]{hyperref}
\usepackage{setspace}
\usepackage{amssymb, amsmath, amsthm, graphicx,mathrsfs}
\usepackage{caption,cite}
\usepackage{mathtools,color}

\usepackage{cleveref}
\usepackage{tikz,verbatim}
\usetikzlibrary{shapes,snakes}
\usetikzlibrary{arrows}

\newtheorem{thm}{Theorem}[section]

\newtheorem{lem}[thm]{Lemma}
\newtheorem{conj}[thm]{Conjecture}

\newtheorem{prob}[thm]{Problem}

\newcommand{\C}{\mathcal{C}}

\renewcommand{\P}{\mathcal{P}}
\renewcommand{\L}{\mathcal{L}}
\renewcommand{\l}{\left}
\renewcommand{\r}{\right}

\title{Block Designs with Large Transversal–Block Intersections}
\author{Jiaxi Nie\footnote{School of Mathematics, Georgia Institute of Technology,
Atlanta, GA 30332 USA. {\tt jnie47@gatech.edu}.}}
\date{\today}

\begin{document}

\maketitle

\begin{abstract}
A block design is a set family, whose members are called blocks, where every pair of points in the universe is covered by exactly one block. Erd\H{o}s asked the following question: is it true that if the universe is of size $n$ and each block is of size $\Omega(\sqrt{n})$, then there exists a transversal intersecting every block in $O(1)$ points? We provide a negative answer to this question. Indeed, for arbitrarily large $n$, we construct such a block design where every transversal intersects some block in $\Omega(\frac{\log n}{\log\log n})$ points.
\end{abstract}

\section{Introduction}
Let $S$ be a finite set and $A_1,\dots, A_m$ its subsets called blocks. We say $(S;A_1,\dots, A_m)$ is a \emph{pairwise balanced block design} if every pair of distinct elements of $X$ is contained in exactly one block. In this paper, we simply call them block designs for brevity. A \emph{transversal} of the block design is a subset of $S$ with non-empty intersection with every $A_i$.  Erd\H{o}s~\cite{erdHos1981combinatorial} asked the following problem.

\begin{prob}\label{prob:main}
Is it true that to every $c_1$ there is a $c_2$ so that if $(S;A_1,\dots,A_m)$ balanced block design satisfying $|S|=n$ and $|A_i|>c_1n^{1/2}$ for every $i$ then there is a transversal $Q$ such that
$$
|Q\cap A_i|\le c_2, ~~~~i=1,\dots,m.
$$
\end{prob}

The stronger assertion for partial designs was disproved by Alon~\cite{alon2026problems}. In this paper, we give a negative answer to Problem~\ref{prob:main}.
\begin{thm}\label{thm:main}
For any integer $n_0$ there exists $n\ge n_0$, and a block design $(S; A_1,\dots, A_m)$ such that $|S|=n$, $|A_i|>0.6\sqrt{n}$ for every $i$ and
for every transversal $Q$ there exists an $i$ such that
$$
|Q\cap A_i|\ge \frac{\log n}{100\log\log n}.
$$
\end{thm}

\section{Preliminary}
A key component of our proof is the hypergraph container lemma developed by Balogh-Morris-Samotij~\cite{balogh2015independent} and Saxton-Thomason~\cite{saxton2015hypergraph} independently. In particular, we make use for the following efficient container lemma by Balogh and Samotij~\cite{balogh2019efficient}. For any hypergraph $H$, we denote its number of vertices and edges by $v(H)$ and $e(H)$ respectively. For any $X\subset V(H)$, let $d_H(X)$ be the number of edges of $H$ containing $X$. For any integer $t$, let $\Delta_t(H)$ be the maximum $d_H(X)$ where $|X|=t$.

\begin{thm}[A simplified version of Theorem 1.1 in~\cite{balogh2019efficient}]\label{thm:HCL}
Let $k$ be a positive integer and $H$ be a nonempty $k$-uniform hypergraph on $n$ vertices. Suppose $\tau\in (0,1)$ and $\eta>0$ are such that
\begin{equation}\label{eq:HCL1}
\tau n\ge 10^8 k^6\eta  
\end{equation}
and
\begin{equation}\label{eq:HCL2}
\Delta_t(H)\le \eta\l(\frac{\tau}{10^6k^5}\r)^{t-1}\frac{e(H)}{n},~~~t=1,\dots, k.    
\end{equation}
Then there exists a family $\C$ such that
\begin{itemize}
    \item[(a)] for every independent set $I$ of $H$, there exists $C\in\C$ such that $I\subset C$;
    \item[(b)] $|\C|\le \exp(\log(\frac{e}{\tau})\tau n)$;
    \item[(c)] for every $C\in \C$, $|C|\le (1-\delta+\tau)n$ where $\delta=\frac{1}{10^3k^4\eta}$.
\end{itemize}
\end{thm}

Another important tool is the Janson inequality~\cite{janson1988exponential}. In particual we use the following version as presented in Alon-Spencer~\cite{AlonSpencer2016}. 
\begin{thm}[Theorem 8.1.1 and Theorem 8.1.2 in~\cite{AlonSpencer2016}]\label{thm:janson}
Let $\Omega$ be a finite universal set, $Y$ a random subset of $\Omega$ given by $\Pr[r\in R]=p_r$, these events being mutually independent over $r\in \Omega$. Let $\{A_i\}_{i\in I}$ be a finite collection of subsets of $\Omega$. Let $B_i$ be the event $A_i\subseteq Y$. For $i,j\in I$, write $i\sim j$ if $i\not= j$ and $A_i\cap A_j\not=\emptyset$. Set
$$
\mu=\sum_{i\in I}\Pr[B_i]~~~~\text{and}~~~~\Delta=\sum_{i\sim j}\Pr[B_i\wedge B_j]
$$
where the sum for $\Delta$ is over ordered pairs. Then
\begin{equation}\label{eq:Janson1}
\Pr\l[\bigwedge_{i\in I} \overline{B_i}\r]\le \exp(-\mu+\frac{\Delta}{2}).    
\end{equation}
Further, if $\Delta\ge \mu$, then 
\begin{equation}\label{eq:Janson2}
\Pr\l[\bigwedge_{i\in I} \overline{B_i}\r]\le \exp\l(-\frac{\mu^2}{2\Delta}\r).    
\end{equation}
\end{thm}

\section{Proof of the main result}
In the rest of the paper, we stick to the following assumptions and notations: Let $q$ be a sufficiently large prime power. We use $\P$ to denote the set of all points of $PG(2,q)$ and use $\L$ to denote the set of all lines of $PG(2,q)$. Let $R$ be a random subset of $\P$ obtained by picking every point randomly and independently with probability 1/2.

We first outline the main idea. Note that $R$ together with the blocks $\{L\cap R\}_{L\in\L}$ form a block design. By the Chernoff bound and the union bound, with high probability every $L\in\L$ has substantial intersection with $R$, and hence every transversal of $R$ is also a transversal of $PG(2,q)$. Let $k=\Theta(\frac{\log q}{\log\log q})$. We say a set $X\subseteq \P$ is an \emph{$k$-arc} if $|X\cap L|< k$ for every $L\in\L$. 

The first step of our proof is to construct a collection of containers for $k$-arcs, i.e. a collection $\C$ of subsets of $\P$ such that for every $k$-arc there exists a $C\in\C$ such that $X\subseteq C$. We make sure that both the number of containers and the size of each container are ``small''.

The second step is to show that, for every container $C$, the probability that $R\cap C$ is a transversal of $PG(2,q)$ is so ``tiny'' that we may confirm, by a union bound, that there exists a deterministic $R$ such that $R\cap C$ is not a transversal for every container $C\in\C$. Therefore, every $k$-arc in $R$ is contained in some $C\cap R$ which is not a transversal, showing that every $k$-arc is not a transversal, which completes the proof.

\subsection{Containers for $k$-arcs}

We will construct a rooted tree where the root is $\P$ and every node in the tree is a subset of $\P$ such that for every non-leaf node $X$, every $k$-arc in $X$ is contained in some child node of $X$. Thus, every $k$-arc of $PG(2,q)$ is contained in some leaf node of the tree. The leaf nodes are our containers.

To control the number of leaf nodes and the size of each leaf node, it suffices to control for a given non-leaf node the number of its children and the size of each of its child. The following is the key lemma.

\begin{lem}\label{lem:k-arc-one-setp-HCL}
Let $\gamma$ and $k$ be positive integers such that $2\le k\le \min\{\gamma,\frac{\log q}{20\log\log q}\}$. Let $X\subseteq \P$ be such that $2\gamma q\le |X|\le (\gamma+1)2q$. Then there exists a collection $\C_X$ of subsets of $X$ such that
\begin{itemize}
    \item[(a)] for every $k$-arc $Y$ in $X$ there exists $C\in\C_X$ such that $Y\subseteq C$;
    \item[(b)] $|\C_x|\le \exp\l(10^{7}k^{6}\log q \cdot q^{1-\frac{1}{k}}\r)$;
    \item [(c)] For every $C\in\C_x$, $|C|\le (1-\frac{1}{10^5k^5})|X|$.
\end{itemize}
\end{lem}

\begin{proof}
We construct a $k$-uniform hypergraph like this: for each line $L\in\L$, we iteratively extract a subset of size $\gamma$ from $L\cap X$, take all of its size-k subsets as edges, and stop when the remaining vertices in $L\cap X$ is less than $\gamma$. More precisely, for each line $L\in\L$, partition $X\cap L$ into disjoint subsets $S_L^{(1)},S_L^{(2)},\dots, S_L^{(s_L)}$ and $R_L$ arbitrarily such that $|S_L^{(1)}|=\dots=|S_L^{(s_L)}|=\gamma$ and $|R_L|<\gamma$. Let $H'$ be a k-uniform hypergraph with $V(H')=X$ and
$$
E(H')=\bigcup_{L\in\L}\bigcup_{i=1}^{s_L}\binom{S^{(i)}_L}{k}
$$
where $\binom{S^{(i)}_L}{k}$ means the collection of all size-$k$ subsets of $S^{(i)}_L$. 

Consider the point-line incidences $(x,L)$ where $x\in X$, $L\in\L$ and $x$ lies on $L$. Since each point is contained in $q+1$ lines, the number of incidences equals $|X|(q+1)$. Since $|R_L|<\gamma$ for every $L$, the number of $(x,L)$ where $x\in R_L$ is at most $\gamma|\L|=\gamma(q_2+q+1)$. Thus,
$$
e(H')\ge \frac{|X|(q+1)-(q^2+q+1)\gamma}{\gamma}\binom{\gamma}{k}\ge q^2\binom{\gamma}{k}.
$$
For every point $x\in X$, $x$ is contained in $q+1$ line and hence contained at most $q+1$ small blocks $S_L^{(i)}$, which implies
$$
\Delta_1(H')\le (q+1)\binom{\gamma-1}{k-1}.
$$
For every $2\le t\le k$ and every size-$t$ subset of $X$ is contained in at most one $S_L^{(i)}$, and hence
$$
\Delta_{t}(H')\le \binom{\gamma-t}{k-t}.
$$
We want to apply Theorem~\ref{thm:HCL} on $H'$ with parameters $\eta=10k$ and $\tau=\frac{10^6k^6}{\gamma q^{\frac{1}{k}}}$. To this end, we need to check the required conditions \eqref{eq:HCL1} and \eqref{eq:HCL2}. Note that $\tau |X|\ge 2\cdot 10^5k^6 q^{1-\frac{1}{k}}\ge 2\cdot 10^8 k^7$ given that $q$ is sufficiently large, thus \eqref{eq:HCL1} holds. Next,
$$
\frac{\Delta_1(H')|X|}{e(H')}\le \frac{(q+1)\binom{\gamma-1}{k-1}(\gamma+1)2q}{q^2\binom{\gamma}{k}}\le 10k=\eta,
$$
and for $2\le t\le k$,
$$
\frac{\Delta_t(H')|X|}{e(H')\eta}\le \frac{\binom{\gamma-t}{k-t}(\gamma+1)2q}{q^2\binom{\gamma}{k}10k}\le \frac{(k-1)(k-2)\dots(k-t+1)}{(\gamma-1)(\gamma-2)\dots(\gamma-t+1)q}\le \l(\frac{k}{\gamma q^{\frac{1}{k}}}\r)^{t-1}=\l(\frac{\tau}{10^6k^5}\r)^{t-1}.
$$
Thus \eqref{eq:HCL2} holds. By Theorem~\ref{thm:HCL}, there exists a collection $\C_X$ of subsets of $X$ such that
\begin{itemize}
    \item[($a'$)] for every independent set $I$ of $H'$ there exists $C\in\C_X$ such that $I\subseteq C$;
    \item[($b$)] $|\C_X|\le \exp\l(\log(\frac{e}{\tau})\tau |X|\r)\le \exp\l(10^7k^6\log q \cdot q^{1-\frac{1}{k}}\r)$;
    \item[($c$)] for every $C\in\C_X$, $|C|\le (1-\frac{1}{10^4k^5}+\tau)|X|\le (1-\frac{1}{10^5k^5})|X|$,
\end{itemize}
where the last inequality in $(c)$ is from the given condition $k\le \frac{\log q}{20\log\log q}$. Note that every edge of $H'$ is a set of $k$ points on the same line, so a $k$-arc in $X$ must be an independent set of $H'$. Hence $(a')$ implies $(a)$.
\end{proof}

\begin{lem}[Container Lemma for $k$-arc]\label{lem:k-arc-HCL}
Let $k$ be a positive integer such that $2\le k\le \frac{\log q}{20\log\log q}$. There exists a collection $\C$ of subsets of $\P$ such that
\begin{itemize}
    \item[(a)] for every $k$-arc $Y$ of $PG(2,q)$ there exists $C\in\C$ such that $Y\subseteq C$;
    \item[(b)] $|\C|\le \exp\l(10^{12}k^{11}(\log q)^2q^{1-\frac{1}{k}}\r)$;
    \item [(c)] for every $C\in\C$, $|C|\le 2kq$.
\end{itemize}
\end{lem}
\begin{proof}
Construct a rooted tree recursively starting with the root being $\P$. Whenever there is a node $X$ in the current tree with $|X|>2kq$, apply Lemma~\ref{lem:k-arc-one-setp-HCL} on $X$ with $\gamma=\l\lfloor\frac{|X|}{2q}\r\rfloor$ to obtain a collection $\C_X$ such that
\begin{itemize}
    \item[($a'$)] for every $k$-arc $Y$ in $X$ there exists $C\in\C_X$ such that $Y\subseteq C$;
    \item[($b'$)] $|\C_x|\le \exp\l(10^{7}k^{6}\log q \cdot q^{1-\frac{1}{k}}\r)$;
    \item [($c'$)] for every $C\in\C_x$, $|C|\le (1-\frac{1}{10^5k^5})|X|\le \exp(-\frac{1}{10^5k^5})|X|$.
\end{itemize}
Call the tree we obtain in the end the container tree and let $\C$ be the collection of its leaf nodes. By the stopping condition, we know that $|C|\le 2kq$ for every $C\in\C$, confirming $(c)$. By $(a')$, every $k$-arc of $PG(2,q)$ is contained in some $C\in\C$, confirming $(a)$. 

It remains to upper bound $|\C|$. By $(c')$ and the stopping condition, the depth of the tree is at most $10^5k^5\log q$. This together with $(b')$ implies that 
$$
|\C|\le \exp\l(10^{7}k^{6}\log q \cdot q^{1-\frac{1}{k}}\cdot (10^5k^5\log q)\r)=\exp(10^{12}k^{11}(\log q)^2q^{1-\frac{1}{k}}),
$$
confirming $(b)$.
\end{proof}

\subsection{Random subset of a small set can hardly be a transversal}
\begin{lem}\label{lem:transversal probability}
Let $k$ be a positive integer such that $k\le \log q/10$. Let $C\subseteq \P$ with $|C|\le 2kq$. Then
$$
\Pr[C\cap R \text{ is a transversal}]\le \exp\l(-\frac{q}{50k}\r).
$$
\end{lem}

\begin{proof}
Let $\L_S\subseteq\L$ be the collection of lines $L$ such that $|C\cap L|< 4k$, and let $\overline{\L_S}=\L\setminus\L_S$. Since every point is contained in $q+1$ lines, the number of incidences $(x,L)$ with $x\in C$ is $|C|(q+1)\le 2kq(q+1)$. Every line in $\overline{\L_S}$ contributes at least $4k$ incidences. Thus, $|\overline{\L_S}|\le \frac{ 2kq(q+1)}{4k}\le \frac{q(q+1)}{2}$, and hence $|\L_S|\ge q^2+q+1-|\overline{\L_S}|\ge q^2/2$.

For each line $L\in \L_S$, add $4k-|L\cap C|$ new points into $L\cap C$ to produce a new line $\hat{L}$ with $|\hat{L}|=4k$. Let $\hat{L_S}:=\{\hat{L}~|~L\in\L_S\}$ and let $\hat R$ be a random subset of the set of newly added points $\cup_{L\in\L_S}\hat L\setminus L$ obtained by selecting each point randomly and independently with probability 1/2. Let $\Omega:=\cup_{\hat{\L_S}}\hat L$, and let $Y=\Omega\setminus(R\cup \hat R)$; thus $Y$ is a random subset of $\Omega$ obtained by selecting each point randomly and independently with probability 1/2. For each $L\in \L_S$, let $B_L$ be the event that $\hat L\subseteq Y$. Let
$$
\mu=\sum_{L\in\L_S}\Pr[B_L]=|\L_S|2^{-4k}\ge q^22^{-4k-1}.
$$
For $L, L'\in \L_S$, write $L\sim L'$ if $L\cap L'\in C$. Note that each point in $C$ is contained in at most $q+1$ lines in $\L_S$. Thus the number of ordered pairs $(L,L')$ such that $L\sim L'$ is at most $|C|(q+1)q\le 2k(q+1)q^2$. Let
$$
\Delta=\sum_{L\sim L'}\Pr[B_L\wedge B_{L'}]\le k(q+1)q^22^{-8k+2}.
$$

The proof divides into two cases.

\noindent\textbf{Case 1: If $\Delta\ge \mu$}, then by \eqref{eq:Janson2},
$$
\Pr[\bigwedge_{L\in\L_S}\overline{B_L}]\le \exp(-\frac{\mu^2}{2\Delta})\le \exp(-\frac{q^2}{32k(q+1)})\le \exp(-\frac{q}{50k}).
$$

\noindent\textbf{Case 2: If $\Delta<\mu$}, then by \eqref{eq:Janson1},
$$
\Pr[\bigwedge_{L\in\L_S}\overline{B_L}]\le \exp(-\mu+\frac{\Delta}{2})\le \exp(-\frac{\mu}{2})\le \exp(-q^22^{-4k-2})\le \exp(-\frac{q}{50k}),
$$ 
where the last inequality comes from the given condition $k\le \log q/10$.

Note that if $C\cap R$ is a transversal, then none of $B_L$ happens. Therefore, $\Pr[C\cap R \text{ is a transversal}]\le \Pr[\bigwedge_{L\in\L_S}\overline{B_L}]\le \exp(-\frac{q}{50k})$ as desired in either case.
\end{proof}

\subsection{Proof of Theorem~\ref{thm:main}}
\begin{proof}
For any integer $n_0$, pick the prime power $q$ such that $0.49q^2\ge n_0$. Recall that $\P$ is the set of all points and $\L$ is the set of all lines of $PG(2,q)$. Further, $R$ is a random subset of $\P$ obtained by picking every point randomly and independently with probability 1/2. Note that $R$ together with the collection $\{R\cap L\}_{L\in\L}$ form a block design. By the Chernoff bound and the union bound, with high probability, $0.51q^2\ge n=|R|\ge 0.49q^2\ge n_0$ and, for every $L\in\L$, $|R\cap L|\ge 0.49q\ge  0.6\sqrt{n}$.

Let $k=\l\lceil\frac{\log q}{40\log\log q}\r\rceil$. Note that $k\ge \frac{\log n}{100\log\log n}$ with high probability. By Lemma~\ref{lem:k-arc-HCL}, there exists a collection $\C$ of subsets of $\P$ such that
\begin{itemize}
    \item[(a)] for every $k$-arc $Y$ of $PG(2,q)$ there exists $C\in\C$ such that $Y\subseteq C$;
    \item[(b)] $|\C|\le \exp\l(10^{12}k^{11}(\log q)^2q^{1-\frac{1}{k}}\r)$;
    \item [(c)]for every $C\in\C$, $|C|\le 2kq$.
\end{itemize}

By $(a)$, for every $k$-arc $Y$ in $R$ there exists a $C\in\C$ such that $Y\subseteq C\cap R$. By $(c)$ and Lemma~\ref{lem:transversal probability},
$$
\Pr[C\cap R \text{ is a transversal}]\le \exp\l(-\frac{q}{50k}\r).
$$
Note that if $Y\subset C\cap R$ and $C\cap R$ is not a transversal, then $Y$ is not a transversal.
Thus, by $(b)$ and the union bound, 
$$
\begin{aligned}
\Pr[\text{there is a $k$-arc in $R$ which is a transversal}]&\le |\C|\exp\l(-\frac{q}{50k}\r)\\
&\le \exp\l(10^{12}k^{11}(\log q)^2q^{1-\frac{1}{k}}-\frac{q}{50k}\r)\\
&\le \exp\l(-\frac{q}{100k}\r)\rightarrow0~~~~\text{as }q\rightarrow \infty.
\end{aligned}
$$
There exists a deterministic $R$ such that $n=|R|\ge n_0$, $|R\cap L|\ge 0.6 \sqrt n$ for every $L\in \L$, and every $k$-arc in $R$ is not a transversal where $k\ge \frac{\log n}{100\log\log n}$.
\end{proof}

\section{Concluding Remarks}
How tight is the $\frac{\log n}{\log\log n}$ lower bound in Theorem~\ref{thm:main}? Indeed, for any block design of size $n$ with every block of size $\Theta(\sqrt{n})$, sampling a random subset with probability $\Omega(\frac{\log n}{\sqrt{n}})$ would produce a transversal intersecting every line in at most $O(\log n)$ points. Therefore, the lower bound in Theorem~\ref{thm:main} is tight up to a $\Theta(\log\log n)$ factor. We conjecture that the $O(\log n)$ upper bound is closer to the truth.

\begin{conj}\label{conj:log}
There exist constants $c>0$ such that the following holds. For any integer $n_0$ there exists $n\ge n_0$, and a block design $(S; A_1,\dots, A_m)$ such that $|S|=n$, $|A_i|>c\sqrt{n}$ for every $i$ and
for every transversal $Q$ there exists an $i$ such that
$$
|Q\cap A_i|\ge c\log n.
$$
\end{conj}

Alon~\cite{alon2026problems} proposed the following conjecture as an alternative route to provide a negative answer to Problem~\ref{prob:main}.
\begin{conj}[Conjecture 4.7 in~\cite{alon2026problems}]
Let $R$ be a random subset of $\P$, the set of points of $PG(2,q)$, obtained by selecting each point randomly and independently with probability $1/2$. Then there exists a function $f$ with $f(q)\rightarrow\infty$ as $q\rightarrow\infty$ such that with high probability $\frac{|B|}{q}\ge f(q)$ for every transversal $B$ of $R$. In fact, this may even be true with $f(q)=\Omega(\log q)$.
\end{conj}
This conjecture remains open. Further, the stronger assertion $f(q)=\Omega(\log q)$ would imply Conjecure~\ref{conj:log}.

\section*{Acknowledgment}
The author thanks Jozsef Balogh for bringing this problem to his attention at IASM in Hangzhou in the summer of 2025.

\section*{AI Declaration}
The author used ChatGPT 6 Astra to obtain a preliminary proof of Lemma~\ref{lem:transversal probability}, which was later simplified and rewritten by the author. The author also used ChatGPT 6 Astra to proofread and polished the manuscript.

\bibliographystyle{abbrv}
\bibliography{refs}

\end{document}